\documentclass[12pt]{article}
\usepackage{amssymb}
\usepackage{amssymb}
\usepackage{amsmath, amscd}
\usepackage{color}
\usepackage{latexsym}
\usepackage{graphicx}
\newtheorem{theorem}{Theorem}[section]
\newtheorem{lemma}[theorem]{Lemma}
\newtheorem{corollary}[theorem]{Corollary}
\newtheorem{definition}[theorem]{Definition}

\newtheorem{proposition}[theorem]{Proposition}

\newcommand{\qed}{\hfill $\Box$ }

\begin{document}

\title{\Large {\bf Outerplane bipartite graphs with isomorphic resonance graphs }}

\maketitle

\begin{center}
{\large \bf Simon Brezovnik$^{a,b}$, Zhongyuan Che$^{c}$,
Niko Tratnik$^{b,d}$,\\ Petra \v Zigert Pleter\v sek$^{d,e}$
}
\end{center}
\bigskip\bigskip

\baselineskip=0.20in

\smallskip

$^a$ {\it University of Ljubljana, Faculty of Mechanical Engineering, Slovenia} \\

$^b$ {\it Institute of Mathematics, Physics and Mechanics, Ljubljana, Slovenia} \\

$^c$ {\it Department of Mathematics, Penn State University, Beaver Campus, Monaca, USA} \\

$^d$ {\it University of Maribor, Faculty of Natural Sciences and Mathematics, Slovenia} \\
\medskip

$^e$ {\it University of Maribor, Faculty of Chemistry and Chemical Engineering, Slovenia}\\
\medskip

\begin{center}
{\tt simon.brezovnik@fs.uni-lj.si, zxc10@psu.edu,\\ niko.tratnik@um.si, petra.zigert@um.si}
\end{center}


\begin{abstract}

We present novel results related to isomorphic resonance graphs of 2-connected outerplane bipartite graphs. 
As the main result, we provide a structure characterization 
for 2-connected outerplane bipartite graphs with isomorphic resonance graphs.
Moreover, two additional characterizations are  expressed in terms of resonance digraphs and via local structures of inner duals of 
2-connected outerplane bipartite graphs, respectively.

\vskip 0.2in
\noindent {\emph{keywords}}: isomorphic resonance graphs, 2-connected outerplane bipartite graph, resonance digraph, inner dual

\end{abstract}

\section{Introduction}

Resonance graphs reflect interactions between perfect matchings (in chemistry known as Kekul\' e structures) of plane bipartite graphs. 
These graphs were  independently introduced by   chemists 
 (El-Basil \cite{el-basil-93/1,el-basil-93/2},  Gr{\"u}ndler \cite{grundler-82}) and also by 
 mathematicians (Zhang, Guo, and Chen \cite{zhgu-88}) under the name 
 $Z$-transformation graph.  {Initially}, resonance graphs were investigated on hexagonal systems \cite{zhgu-88}. Later, this concept was generalized to  plane
  bipartite graphs, see \cite{Z06,ZZ97,ZZ00,ZZY04}. 
 
 In recent years, various structural   properties of resonance graphs of  (outer)plane bipartite  graphs were obtained \cite{C18,C19,C20,C21}. 
{ The problem of characterizing 2-connected outerplane bipartite graphs with isomorphic resonance graphs is interesting and nontrivial. 
There are outerplane bipartite graphs $G$ and $G'$ whose inner duals are isomorphic paths but with non-isomorphic resonance graphs. 
For example, let $G$ be a linear benzenoid chain (a chain in which every non-terminal hexagon is linear) with $n$ hexagons, 
and  let $G'$ be a fibonaccene (a benzenoid chain in which every non-terminal hexagon {is} angular, see \cite{KZ05})  with $n$ hexagons, where $n > 2$. Then the inner dual $T$ of graph $G$ is isomorphic to the inner dual $T'$ of graph $G'$, since $T$ and $T'$ are both paths on $n$ vertices.
However, their resonance graphs $R(G)$ and $R(G')$ are not isomorphic: $R(G)$ is a path and $R(G')$ is a Fibonacci cube, see Figure \ref{figure1}. 

\begin{figure}[h!] 
\begin{center}
\includegraphics[scale=0.8, trim=0cm 0.5cm 0cm 0cm]{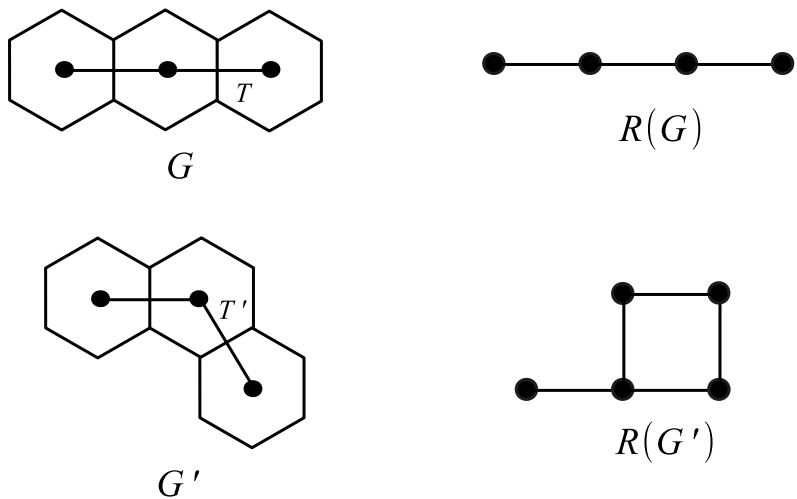}
\end{center}
\caption{\label{figure1} Resonance graphs of the linear benzenoid chain and fibonaccene with three hexagons.}
\end{figure}

In \cite{br-tr-zi,br-tr-zi-1} the problem of finding catacondensed even ring systems (shortly called CERS)  with isomorphic resonance graphs was investigated. 
More precisely, the relation of evenly homeomorphic CERS was introduced and it was proved that if two CERS are evenly homeomorphic, 
then  {their} resonance graphs are isomorphic. 
 Conversely, it is true for catacondensed even ring chains but not for all CERS \cite{br-tr-zi-1}. 
Moreover, in \cite{br-tr-zi-2} it was proved that if two 2-connected outerplane bipartite graphs are evenly homeomorphic,
 then  their resonance graphs are isomorphic. In papers \cite{br-tr-zi-1,br-tr-zi-2}, the following open problem was stated.
\bigskip
 
\noindent
\textbf{Problem 1.} \textit{Characterize 2-connected outerplane bipartite graphs  with isomorphic resonance graphs.}

In this paper we solve the above problem. Firstly, 
we state all the  {needed} definitions and previous results {as preliminaries}. 
The main result, Theorem \ref{glavni1}, is presented in Section \ref{sec3}. The necessity part of this result is stated as Theorem \ref{glavni}. Moreover, in Corollary \ref{cor1} we show that two 2-connected outerplane bipartite graphs have isomorphic resonance graphs if and only if they can be properly two colored so that their resonance digraphs are isomorphic. 
In addition, {by Corollary \ref{cor2} it} follows that 2-connected outerplane bipartite graphs $G$ and $G'$ have isomorphic resonance graphs if and only if there exists an isomorphism $\alpha$ between their inner duals $T$ and $T'$ such that for any 3-path $xyz$ of $T$, the triple $(x, y, z)$ is regular if and only if $(\alpha(x), \alpha(y), \alpha(z))$ is regular.

\section{Preliminaries}

 We say that two faces of a plane graph $G$ are \textit{adjacent} if they have an edge in common. 
 An \textit{inner face} { (also called a \textit{finite face})} adjacent to the \textit{outer face} {(also called the \textit{infinite face})}  is named a \textit{peripheral face}. 
 In addition, we denote the set of edges lying on some face $s$ of $G$ by $E(s)$. The subgraph induced by the edges in $E(s)$ is the \textit{periphery of $s$} and denoted by $\partial s$. The periphery of the outer face is also called the \textit{periphery of $G$} and denoted by $\partial G$. Moreover, for a  peripheral face $s$ and the outer face $s_0$, the subgraph induced by the edges in $E(s) \cap E(s_0)$ is called the {\em common periphery} of $s$ and $G$, and denoted by $\partial s \cap \partial G$. The vertices of $G$ that belong to the outer face are called \textit{peripheral vertices} and the remaining vertices are \textit{interior vertices}.  Furthermore, an \textit{outerplane graph} is a plane graph in which all vertices are peripheral vertices.
\smallskip

 A bipartite graph $G$ is \textit{elementary} if and only if it is connected and each edge is contained in some perfect matching of
$G$. Any elementary bipartite graph other than $K_2$ is 2-connected. 
Hence, if $G$ is a plane elementary bipartite graph with more than two vertices, then the periphery of each face of $G$ is an even cycle.
A peripheral face $s$ of a plane elementary bipartite graph $G$ is called \textit{reducible} if the subgraph $H$ of $G$ obtained by
removing all internal vertices (if exist) and edges on the common periphery of $s$ and $G$ is elementary. 
\smallskip

The \textit{inner dual} of a plane graph $G$ is a graph whose vertex set is the set of all inner faces of $G$, and two vertices being adjacent if the corresponding faces are adjacent.
\smallskip

A \textit{perfect matching} $M$ of a graph $G$ is a set of independent edges of $G$ such that every vertex of $G$ is incident with exactly one edge from $M$. An even cycle $C$ of $G$ is called \textit{$M$-alternating} if the edges of $C$ appear alternately in $M$ and in $E(G) \setminus M$. Also, a face $s$ of a 2-connected plane bipartite graph is \textit{$M$-resonant} if $\partial s$ is an $M$-alternating cycle.
\smallskip

Let $G$ be a plane elementary bipartite graph and $\mathcal{M}(G)$ be the set of all perfect matchings of $G$. 
Assume that $s$ is a reducible face of  $G$. 
By \cite{TV12},  the common periphery of $s$ and $G$ is an odd length path $P$. 
By Proposition 4.1 in \cite{C18},  $P$ is $M$-alternating for any perfect matching $M$ of $G$,
and $\mathcal{M}(G)=\mathcal{M}(G; P^{-}) \cup\mathcal{M}(G; P^{+})$,
where  $\mathcal{M}(G; P^{-})$ {is} the set of perfect matchings $M$ of $G$
such that two end edges of $P$ are not contained in $M$ or $P$ is a single edge and not contained in $M$;
$\mathcal{M}(G; P^{+})$ is the set of perfect matchings $M$ of $G$ such that two end edges of $P$ are contained in $M$
or $P$ is a single edge and contained in $M$.
Furthermore, $\mathcal{M}(G; P^{-})$ and $\mathcal{M}(G; P^{+})$ can be partitioned as  
\begin{eqnarray*}
\mathcal{M}(G; P^{-})&=&\mathcal{M}(G; P^{-}, \partial s) \cup\mathcal{M}(G; P^{-}, \overline{\partial s})\\
\mathcal{M}(G; P^{+})&=&\mathcal{M}(G; P^{+}, \partial s) \cup\mathcal{M}(G; P^{+}, \overline{\partial s})
\end{eqnarray*}
where $\mathcal{M}(G; P^{-}, \partial s)$ (resp., $\mathcal{M}(G; P^{-}, \overline{\partial s})$) 
is the set of perfect matchings $M$ in $\mathcal{M}(G; P^{-})$ such that $s$ is $M$-resonant (resp., not $M$-resonant),
and $\mathcal{M}(G; P^{+}, \partial s)$ (resp., $\mathcal{M}(G; P^{+}, \overline{\partial s})$) 
is the set of perfect matchings $M$ in $\mathcal{M}(G; P^{+})$ such that $s$ is $M$-resonant (resp., not $M$-resonant).
\smallskip

Let $G$ be a  plane bipartite graph with a perfect matching.  The {\em resonance graph} (also called \textit{$Z$-transformation graph}) $R(G)$ of $G$ is the graph whose vertices are the  perfect matchings of $G$, and two perfect matchings $M_1,M_2$ are adjacent whenever their symmetric difference forms the edge set of exactly one inner face $s$ of $G$. In this case, we say that the edge $M_1M_2$ has the \textit{face-label} $s$.

{Let $H$ and $K$ be two graphs with vertex sets $V(H)$ and $V(K)$, respectively.
The Cartesian product of $H$ and $K$  is a graph with the vertex set $\{(h,k) \mid h \in V(H), k \in V(K)\}$
such that two vertices $(h_1,k_1)$ and $(h_2,k_2)$ are adjacent if either $h_1h_2$ is an edge of 
$H$ and $k_1=k_2$ in $K$ or $k_1k_2$ is an edge of $K$ and $h_1=h_2$ in $H$.
Assume that $G$ is a disjoint union of two plane bipartite graphs $G_1$ and $G_2$. 
Then by definitions, the resonance graph $R(G)$
is the Cartesian product of $R(G_1)$ and $R(G_2)$.}
\smallskip

Assume that $G$ is a plane bipartite graph whose vertices are properly colored black and white 
such that adjacent vertices receive different colors. Let $M$ be a perfect matching of $G$.
An $M$-alternating cycle $C$ of $G$ is \textit{$M$-proper} (resp., \textit{$M$-improper}) 
if every edge of $C$ belonging to $M$ goes from white to black vertex (resp., from black to white vertex) along the clockwise orientation of $C$.
A plane elementary bipartite graph $G$ with a perfect matching has a unique perfect matching $M_{\hat{0}}$ (resp., $M_{\hat{1}}$)
such that $G$ has no proper $M_{\hat{0}}$-alternating cycles (resp., no improper $M_{\hat{1}}$-alternating cycles) \cite{ZZ97}.

The  \textit{resonance digraph}, denoted by $\overrightarrow{R}(G)$,
is the digraph obtained from $R(G)$ by adding a direction for each edge  so that
$\overrightarrow{M_1M_2}$ is a directed edge from $M_1$ to $M_2$ if $M_1 \oplus M_2$ 
is a proper $M_1$-alternating (or, an improper $M_2$-alternating) cycle surrounding an inner face of $G$.
Let $\mathcal{M}(G)$ be the set of all perfect matchings of $G$.
Then a partial order $\le$ can be defined on $\mathcal{M}(G)$ such that $M' \le M$
if there is a directed path from $M$ to $M'$ in $\overrightarrow{R}(G)$. 
{When $G$ is a plane elementary bipartite graph,  $\mathcal{M}(G)$
is a finite distributive lattice whose Hasse diagram is isomorphic to $\overrightarrow{R}(G)$ \cite{LZ03}.}
It is well known that $M_{\hat{0}}$ is the minimum and $M_{\hat{1}}$ the maximum of the distributive lattice  $\mathcal{M}(G)$ \cite{LZ03,ZLS08}.
\smallskip

We now present the concept of a reducible face decomposition, see \cite{ZZ00} and \cite{C18,C19}. Firstly, we introduce the \textit{bipartite
ear decomposition} of a plane elementary bipartite graph $G$ with $n$ inner faces.  
Starting from an edge $e$ of $G$, join its two end vertices by a path $P_1$ of odd length and proceed inductively to build a sequence of bipartite
graphs as follows. If $G_{i-1} = e + P_1 + \cdots + P_{i-1}$ has already been constructed, add the $i$th ear $P_i$ of odd length by joining any
two vertices belonging to different  bipartition sets of $G_{i-1}$ such that $P_i$ has no internal vertices in common with the vertices of $G_{i-1}$. A bipartite
ear decomposition of a plane elementary bipartite graph $G$ is called a \textit{reducible face decomposition} (shortly $RFD$) if $G_1$ is a
periphery of an inner face $s_1$ of $G$, and the $i$th ear $P_i$ lies in the exterior of $G_{i-1}$ such that $P_i$ and a part of the periphery of $G_{i-1}$
surround an inner face $s_i$ of $G$ for all $i \in \{ 2, \ldots, n \}$. For such a decomposition, we use notation $RFD(G_1, G_2, \ldots, G_n)$, where $G_n=G$.
{It was shown \cite{ZZ00} that a plane bipartite graph with more than two vertices is elementary if and only if it has a reducible face decomposition.}
\smallskip

Let $H$ be a convex subgraph of a graph $G$. The \textit{peripheral convex expansion} of $G$ with respect to $H$,
denoted by $pce(G;H)$, is the graph obtained from $G$ by the following procedure:
\begin{enumerate}
\item [$(i)$] Replace each vertex $v$ of $H$ by an edge $v_1v_2$.
\item [$(ii)$] Insert edges between $v_1$ and the neighbours of $v$ in $V(G) \setminus V(H)$.
\item [$(iii)$] Insert the edges $u_1v_1$ and $u_2v_2$ whenever $u,v$ of $H$ are adjacent in $G$.
\end{enumerate}

Two edges $uv$ and $xy$ of a connected graph $G$ are said to be in \textit{relation $\Theta$} (also known as \textit{Djokovi\' c-Winkler relation}), 
denoted by $uv \Theta xy$, if $d_G(u, x) + d_G(v, y) \neq  d_G(u, y) + d_G(v, x)$. 
It is well known that if $G$ is a plane elementary bipartite graph, then its resonance graph $R(G)$ is a median graph \cite{ZLS08} and therefore, 
the relation $\Theta$  is an equivalence relation on the set of edges $E(R(G))$.

Let $xy$ be an edge of a resonance graph $R(G)$ and
$F_{xy}=\{e \in E(R(G)) \mid e \Theta xy\}$ be the set of all edges in relation $\Theta$ with $xy$ in $R(G)$, where $G$ is a plane elementary bipartite graph.  
By Proposition 3.2 in \cite{C18}, all edges in $F_{xy}$ have the same face-label.
On the other hand, two edges with the same face-label can be in different $\Theta$-classes of $R(G)$.

\smallskip

We now present several results from previous papers which will be needed later.

\begin{proposition} \cite{TV12} \label{pro-ves} 
Let $G$ be a plane elementary bipartite graph other than $K_2$. Then the outer cycle of $G$ is improper $M_{\hat{0}}$-alternating 
as well as proper $M_{\hat{1}}$-alternating, where  $M_{\hat{0}}$ is the minimum and $M_{\hat{1}}$ the maximum in the finite distributive lattice  $\mathcal{M}(G)$.
\end{proposition}

The induced subgraph of a graph $G$ on $W \subseteq V(G)$  will be denoted as $\langle W \rangle$.
\begin{theorem}\label{T:MedianZ(G)}\cite{C18}
Assume that $G$ is a plane elementary bipartite graph and $s$ is a reducible face of $G$. 
Let $P$ be the common periphery of $s$ and $G$.
Let $H$ be the subgraph of $G$ obtained by removing all internal vertices and edges of $P$.
Assume that $R(G)$ and $R(H)$ are resonance graphs of $G$ and $H$ respectively. 
Let $F$ be the set of all edges in $R(G)$ with the face-label $s$.
Then $F$ is a $\Theta$-class of $R(G)$ and $R(G)-F$ has exactly two components
$\langle \mathcal{M}(G; P^{-}) \rangle$ and $\langle \mathcal{M}(G; P^{+}) \rangle$.
Furthermore, 

(i) $F$ is a matching defining an isomorphism between 
$\langle \mathcal{M}(G; P^{-}, \partial s) \rangle$ and $\langle \mathcal{M}(G; P^{+}, \partial s) \rangle$;
 
(ii) $\langle \mathcal{M}(G; P^{-}, \partial s) \rangle$ is convex in $\langle \mathcal{M}(G; P^{-}) \rangle$, 
$\langle \mathcal{M}(G; P^{+}, \partial s) \rangle$ is convex in $\langle \mathcal{M}(G; P^{+}) \rangle$;

(iii) $\langle \mathcal{M}(G; P^{-}) \rangle$ and $\langle \mathcal{M}(G; P^{+}) \rangle$ are median graphs, 
where $\langle \mathcal{M}(G; P^{-}) \rangle \cong R(H)$.

In particular, $R(G)$ can be obtained from $R(H)$ by a peripheral convex expansion 
if and only if $\mathcal{M}(G; P^{+})=\mathcal{M}(G; P^{+}, \partial s)$.
\end{theorem}

\begin{proposition}\label{P:OuterPlaneReducibleFace}\cite{C19}
Let $G$ be a 2-connected outerplane bipartite  graph.
Assume that $s$ is a reducible face of $G$. Then $s$ is adjacent to exactly one inner face of $G$.
\end{proposition}

For any 2-connected outerplane bipartite graph $G$ and a reducible face $s$ of $G$,
we know  from \cite{TV12} that the common periphery of $s$ and $G$ is an odd length path $P$.
By Proposition \ref{P:OuterPlaneReducibleFace}, $s$ is adjacent to exactly one inner face $s'$ of $G$.
It is clear that the common edge of $s$ and $s'$ is a single edge $e$. 
Therefore, $E(s)=e \cup E(P)$ and the odd length path $P$ must have at least three edges.

\begin{theorem} \cite{C19} \label{th0} 
Let $G$ be a 2-connected outerplane bipartite graph. Assume that $s$ is a reducible face of $G$ and $P$ is the common
periphery of $s$ and $G$. Let $H$ be the subgraph of $G$ obtained by removing all internal vertices and edges of $P$. 
Then $R(G)$ can be obtained from $R(H)$ by a peripheral convex expansion, 
that is, $R(G) = pce(R(H); T)$ where the set of all edges between $R(H)$ and
$T$ is a $\Theta$-class of $R(G)$ with the face-label $s$. 
Moreover,
\begin{itemize}
\item [$(i)$] $R(G)$ has exactly one more $\Theta$-class than $R(H)$ and it has the face-label $s$, and
\item [$(ii)$] each of other $\Theta$-classes of $R(G)$ can be obtained from the corresponding $\Theta$-class of $R(H)$ with the same face-label (adding
more edges if needed).
\end{itemize}
\end{theorem}

\begin{theorem}\label{T:OuterPlanePeripheryExpansionSequence}\cite{C19}
Let $G$ be a 2-connected outerplane bipartite  graph and $R(G)$ be its resonance graph.
Assume that $G$ has a reducible face decomposition $G_i (1\le i \le n)$  where $G_n=G$
associated with a sequence of inner faces $s_i (1\le i \le n)$ and a sequence of odd length ears $P_i (2 \le i \le n)$.
Then $R(G)$ can be obtained from the one edge graph by a sequence of peripheral convex expansions
with respect to the above reducible face decomposition of $G$. 
Furthermore,  $R(G_1)=K_2$ where the edge has the face-label $s_1$;  
for $2 \le i \le n$, $R(G_{i})=pce(R(G_{i-1}); T_{i-1})$ 
where the set of all edges between $R(G_{i-1})$ and $T_{i-1}$  is a $\Theta$-class in $R(G_{i})$ 
with the face-label $s_i$, $R(G_i)$ has exactly one more $\Theta$-class than $R(G_{i-1})$ and it has the face-label $s_i$,
each of other $\Theta$-classes of $R(G_i)$ can be obtained from the corresponding $\Theta$-class of $R(G_{i-1})$ 
with the same face-label (adding more edges if needed).
\end{theorem}

The \textit{induced graph $\Theta(R(G))$} on the $\Theta$-classes of $R(G)$ is a graph
whose vertex set is the set of $\Theta$-classes, and two vertices $E$ and $F$ of $\Theta(R(G))$ are adjacent if  
$R(G)$ has two incident edges $e \in E$ and $f \in F$ such that $e$ and $f$  are not contained in a common 4-cycle of $R(G)$.
It is well-known that if $s$ and $t$ are two face labels of incident edges of a 4-cycle of $R(G)$, then $s$ and $t$ are vertex disjoint in $G$
and $M$-resonant for a perfect matching $M$ of $G$;  
if $s$ and $t$ are two face labels of incident edges not contained in a common 4-cycle of $R(G)$, then $s$ and $t$ are adjacent in $G$
and $M$-resonant for a perfect matching $M$ of $G$.

\begin{theorem}\label{T:Theta(R(G))-Isomorphic}\cite{C19}
Let $G$ be a 2-connected outerplane bipartite  graph and $R(G)$ be its resonance graph. 
Then the graph $\Theta(R(G))$ induced by the $\Theta$-classes of $R(G)$ is a tree 
and isomorphic to the inner dual of $G$.
\end{theorem}

\section{Main results} 
\label{sec3}

In this section, we characterize resonance graphs of 2-connected outerplane bipartite graphs with isomorphic resonance graphs. 
 We start with the following lemma, which is a more detailed version 
 of Theorem \ref{th0}  \cite{C19} and Lemma 1 \cite{C21} for 2-connected outerplane bipartite graphs.
We use  $\mathcal{M}(G; e)$  to denote the set of perfect matchings of a graph $G$ 
containing the edge $e$ of $G$.

\begin{lemma} \label{grozna} Let $G$ be a 2-connected outerplane bipartite graph. 
Assume that $s$ is a reducible face of $G$, $P$ is the common
periphery of $s$ and $G$ and $e \in E(s)$ {is} a unique edge that does not belong to $P$. 
Let $H$ be the subgraph of $G$ obtained by removing all internal vertices and edges of $P$.

Further, assume that $H$ has more than two vertices. 
Let $M_{\hat{0}}$ be the minimum and $M_{\hat{1}}$ {be} the maximum in the distributive lattice $\mathcal{M}(H)$.  
Then $e$ is contained in exactly one of $M_{\hat{0}}$ and $M_{\hat{1}}$.

\begin{enumerate}
\item 
[$(i)$] Suppose that $M_{\hat{0}} \notin \mathcal{M}(H; e)$. 
Let $\widehat{M_{\hat{0}}}$ be the perfect matching of $G$ such that $M_{\hat{0}} \subseteq \widehat{M_{\hat{0}}}$ 
and  $\widehat{M_{\hat{1}}}$ {be} the perfect matching of $G$ such that $M_{\hat{1}} \setminus \{ e \} \subseteq \widehat{M_{\hat{1}}}$.
Then $\widehat{M_{\hat{0}}} \in \mathcal{M}(G;P^-, \overline{\partial s})$  is the minimum,  
and $\widehat{M_{\hat{1}}} \in \mathcal{M}(G;P^+, {\partial s})$ is the maximum of the finite distributive lattice $\mathcal{M}(G)$.

\item  
[$(ii)$] Suppose that  $M_{\hat{0}} \in \mathcal{M}(H; e)$. Let $\widehat{M_{\hat{0}}}$ be the perfect matching of $G$ such that $M_{\hat{0}} \setminus \{ e \} \subseteq \widehat{M_{\hat{0}}}$ and  $\widehat{M_{\hat{1}}}$ {be} the perfect matching of $G$ such that $M_{\hat{1}}  \subseteq \widehat{M_{\hat{1}}}$. Then  $\widehat{M_{\hat{0}}} \in \mathcal{M}(G;P^+, {\partial s})$ is the minimum,  
and $\widehat{M_{\hat{1}}} \in \mathcal{M}(G;P^-, \overline{\partial s})$ is the maximum of the finite distributive lattice $\mathcal{M}(G)$.
\end{enumerate}
 
\end{lemma}

\begin{proof}   {Assume that $e$ is the unique edge of $E(s)$ that does not belong to $P$.} 
By Theorem \ref{th0} ,  $R(G)=pce(R(H), \langle\mathcal{M}(H; e)\rangle)$,
where the edges between $R(H)$ and $\langle\mathcal{M}(H; e)\rangle$ is a $\Theta$-class of $R(G)$ with the face-label $s$.
Moreover, $R(H) \cong \langle\mathcal{M}(G;P^-) \rangle$
and $ \langle\mathcal{M}(H; e)\rangle \cong \langle \mathcal{M}(G;P^-, {\partial s}) \rangle \cong \langle \mathcal{M}(G;P^+, {\partial s}) \rangle$. See Figure \ref{figure2}.

\begin{figure}[h!] 
\begin{center}
\includegraphics[scale=0.8, trim=0cm 0.5cm 0cm 0cm]{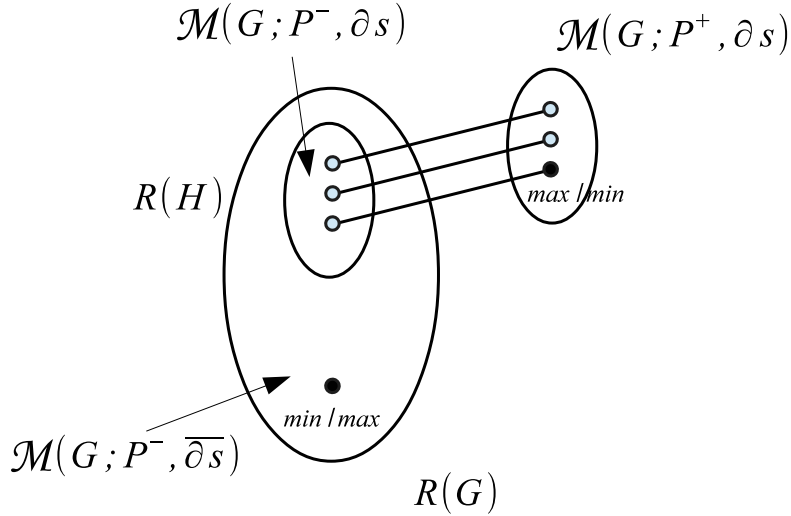}
\end{center}
\caption{\label{figure2} {A} peripheral convex expansion of the resonance graph $R(G)$.}
\end{figure}

Any 2-connected outerplane bipartite graph has two perfect matchings whose edges form alternating edges on the outer cycle of the graph.
By Proposition \ref{pro-ves}, one is the maximum and the other is the minimum in the finite distributive lattice on the set of perfect matchings of the graph.

Let $M_{\hat{0}}$ be the minimum and $M_{\hat{1}}$ {be} the maximum in the finite distributive lattice $\mathcal{M}(H)$.  
Then the outer cycle of $H$ is both {improper} $M_{\hat{0}}$-alternating and {proper} $M_{\hat{1}}$-alternating. 
Note that $e$ is an edge of the outer cycle of $H$.
Then $e$ is contained in exactly one of $M_{\hat{0}}$ and $M_{\hat{1}}$.

We will show only part $(i)$, since the proof of $(ii)$ is analogous. 
Suppose that $M_{\hat{0}}$ does not contain the edge $e$. 
{Recall that} the outer cycle of $H$ is improper $M_{\hat{0}}$-alternating. 
By the definition of $\widehat{M_{\hat{0}}}$, the outer cycle of $G$ is improper $\widehat{M_{\hat{0}}}$-alternating. 
Therefore,  $\widehat{M_{\hat{0}}}$ is the minimum of the distributive lattice $\mathcal{M}(G)$
{since $G$ is an outerplane bipartite graph. 
Note that three consecutive edges  on the periphery of $s$, namely $e$ and two end edges of $P$, are not contained in $\widehat{M_{\hat{0}}}$.
Then $s$ is not $\widehat{M_{\hat{0}}}$-resonant.}
So, $\widehat{M_{\hat{0}}} \in \mathcal{M}(G;P^-, \overline{\partial s})$.

{Note that $M_{\hat{1}}$ contains the edge $e$ since $M_{\hat{0}}$ does not contain $e$ by our assumption for part $(i)$. 
Recall that the outer cycle of $H$ is proper $M_{\hat{1}}$-alternating. 
By the definition of $\widehat{M_{\hat{1}}}$, $\widehat{M_{\hat{1}}} \in \mathcal{M}(G;P^+)$
and the outer cycle of $G$ is again proper $\widehat{M_{\hat{1}}}$-alternating.
It follows that $s$ is  $\widehat{M_{\hat{1}}}$-resonant.}
Consequently, $\widehat{M_{\hat{1}}} \in \mathcal{M}(G;P^+, {\partial s}) $ is the maximum of the finite distributive lattice $\mathcal{M}(G)$. \qed
\end{proof}
\bigskip

To state the next theorem, we need the following notation. Let $G$ and $G'$ be 2-connected outerplane bipartite graphs. Suppose that $\phi$ is an isomorphism  between resonance graphs  $R(G)$ and $R(G')$. 
By Theorem \ref{T:Theta(R(G))-Isomorphic}, {the} isomorphism $\phi$ induces an isomorphism between inner duals of $G$ and $G'$, 
which we denote by $\widehat{\phi}$.


\begin{theorem} \label{glavni} Let $G$ and $G'$ be 2-connected outerplane bipartite graphs. 
If there exists an isomorphism $\phi$ between resonance graphs  $R(G)$ and $R(G')$, then
$G$ has a reducible face decomposition $G_i (1 \le i \le n)$ where $G_n=G$
associated with the face sequence 
$s_i (1 \le i \le n)$ and the odd length path sequence $P_i (2 \le i  \le n)$;  and
$G'$ has a reducible face decomposition $G'_i (1 \le i \le n)$ where $G'_n=G'$
associated with the face sequence $s'_i (1 \le i  \le n)$ and the odd length path sequence $P'_i (2 \le i  \le n)$ satisfying three properties:
\begin{itemize}
\item[(i)] the isomorphism $\widehat{\phi}$ between the inner duals of $G$ and $G'$ maps $s_i$ to $s'_i$  for   $1 \le i  \le n$; 
\item[(ii)] $G$ and $G'$ can be properly two colored so that odd length paths $P_i$ and $P'_i$ 
either both start from a black vertex and end with a white vertex, or both start from 
a white vertex and end with a black vertex in  clockwise orientation along the peripheries of $G_i$ and $G'_i$ for $2 \le i  \le n$;
\item[(iii)] ${\phi}$ is an isomorphism between resonance digraphs $\overrightarrow{R}(G)$ and $\overrightarrow{R}(G')$ with respect to the colorings from property $(ii)$.
\end{itemize}
\end{theorem}


\begin{proof} 
{Let $\phi: R(G) \longrightarrow R(G')$ be an isomorphism between $R(G)$ and  $R(G')$.
By Theorem \ref{T:Theta(R(G))-Isomorphic},
the graph $\Theta(R(G))$ induced by the $\Theta$-classes of $R(G)$ is a tree 
and isomorphic to the inner dual of $G$, 
and  the graph $\Theta(R(G'))$ induced by the $\Theta$-classes of $R(G')$ is a tree 
and isomorphic to the inner dual of $G'$. 
By the peripheral convex expansions with respect to a reducible face decomposition
of a 2-connected outerplane bipartite graph given by Theorem \ref{T:OuterPlanePeripheryExpansionSequence}, 
we can see that $\phi$ induces an isomorphism $\widehat{\phi}$ between the inner duals of $G$ and $G'$.
So, $G$ and $G'$ have the same number of inner faces.}

Suppose that $G$ and $G'$ have $n$ inner faces. 
Obviously, all three properties hold if $n=1$ or $n=2$. 
{Let $n \ge 3$.} We proceed by induction on $n$ and therefore assume 
that all three properties hold for {any} 2-connected outerplane bipartite graphs with less than $n$ inner faces.

Let $s_n$ be a reducible face of $G$, $P_n$ be  the common
periphery of $s_n$ and $G$, and  $E$ be the $\Theta$-class in $R(G)$ corresponding to $s_n$. 
Moreover, we denote by $E'$ the $\Theta$-class in $R(G')$ obtained from $E$ by the isomorphism $\phi$,
and $s'_n$ the corresponding reducible face  of $G'$. 
Then $s'_n=\widehat{\phi}(s_n)$. Also, we denote by $P'_n$  the common periphery of $s'_n$ and $G'$.

By Theorem \ref{th0}, the graph $R(G)$ is obtained from $R(G_{n-1})$ by a peripheral convex expansion with respect to the $\Theta$-class $E$. 
Similarly, the graph $R(G')$ is obtained from $R(G'_{n-1})$ by a peripheral convex expansion with respect to the $\Theta$-class $E'$. 
Since $\phi$ is an isomorphism between $R(G)$ and $R(G')$ such that $\phi$ maps $E$ to $E'$,
it follows that $R(G_{n-1})$ and $R(G'_{n-1})$ are isomorphic and the restriction of $\phi$ on $R(G_{n-1})$ 
is an isomorphism $\phi_{n-1}$ between $R(G_{n-1})$ and $R(G'_{n-1})$. 
{Let $\widehat{\phi}_{n-1}$ be the induced isomorphism between the inner duals of $G_{n-1}$ and $G'_{n-1}$.
Then $\widehat{\phi}_{n-1}$ is the restriction of $\widehat{\phi}$ on the inner dual of $G_{n-1}$.}

Since $G_{n-1}$ and $G'_{n-1}$ have $n-1$ inner faces, by the induction hypothesis $G_{n-1}$ has a reducible face decomposition $G_i (1 \le i \le n-1)$
associated with the face sequence 
$s_i (1 \le i \le n-1)$ and the odd length path sequence $P_i (2 \le i  \le n-1)$;  and
$G'$ has a reducible face decomposition $G'_i (1 \le i \le n-1)$ 
associated with the face sequence $s'_i (1 \le i  \le n-1)$ and the odd length path sequence $P'_i (2 \le i  \le n-1)$ satisfying 
properties $(i)$ $s'_i=\widehat{\phi}_{n-1}(s_i)$ for $1 \le i  \le n-1$, $(ii)$ $G_{n-1}$ and $G'_{n-1}$ can be properly two colored so that odd length paths $P_i$ and $P'_i$ 
either both start from a black vertex and end with a white vertex, or both start from 
a white vertex and end with a black vertex in  clockwise orientation along the peripheries of $G_i$ and $G'_i$ for $2 \le i  \le n-1$, and $(iii)$ ${\phi}_{n-1}$ is an isomorphism between resonance digraphs $\overrightarrow{R}(G_{n-1})$ and $\overrightarrow{R}(G'_{n-1})$ with respect to the colorings from property $(ii)$. 

Obviously, since $s_n'=\widehat{\phi}(s_n)$ and $s'_i=\widehat{\phi}_{n-1}(s_i)=\widehat{\phi}(s_i)$ for $1 \le i  \le n-1$,
property $(i)$ holds for 
{the above} reducible face decompositions of $G$ and $G'$.
It remains to show that the these reducible face decompositions  satisfy property $(ii)$ when $i=n$,  
and $\phi$ is an isomorphism between  resonance digraphs $\overrightarrow{R}(G)$ and $\overrightarrow{R}(G')$
with respect to the colorings from property $(ii)$, that is, property $(iii)$ holds.

By Proposition \ref{P:OuterPlaneReducibleFace},
$s_n$ is adjacent to exactly one inner face of $G$ since $s_n$ is a reducible face of $G$.
Suppose that the unique {inner} face adjacent to $s_n$ is $s_j$. Since $\widehat{\phi}$ 
is an isomorphism between the inner duals of $G$ and $G'$, the unique {inner} face adjacent to $s'_n$ is $s'_j$.
By Lemma \ref{grozna}, $\partial s_n \cap \partial s_{j}$ is an edge $uv$ on $\partial G_{n-1}$,
 and $\partial s'_n \cap \partial s'_{j}$ is an edge $u'v'$ on $\partial G'_{n-1}$.
It is clear that $u$ and $v$ (resp., $u'$ and $v'$) are two end vertices of $P_n$ (resp., $P'_n$).
Moreover,  $R(G)=pce(R(G_{n-1}), \langle\mathcal{M}(G_{n-1}; uv)\rangle)$ 
where $R(G_{n-1}) \cong \langle\mathcal{M}(G;P_n^-) \rangle$ and 
$\langle\mathcal{M}(G_{n-1}; uv)\rangle \cong \langle \mathcal{M}(G;P_n^-, {\partial s_n}) \rangle \cong \langle \mathcal{M}(G;P_n^+, {\partial s_n}) \rangle$,
and $R(G')=pce(R(G'_{n-1}), \langle\mathcal{M}(G'_{n-1}; u'v')\rangle)$
where $R(G'_{n-1}) \cong \langle\mathcal{M}(G';{P'}_n^-) \rangle$ and 
$\langle\mathcal{M}(G'_{n-1}; u'v')\rangle   \cong \langle \mathcal{M}(G';{P'}_n^-, {\partial s'_n}) \rangle \cong \langle \mathcal{M}(G';{P'}_n^+, {\partial s'_n}) \rangle$.

{Recall that $\phi$ is an isomorphism 
between $R(G)=pce(R(G_{n-1}), \langle\mathcal{M}(G_{n-1}; uv)\rangle)$ 
and $R(G')=pce(R(G'_{n-1}), \langle\mathcal{M}(G'_{n-1}; u'v')\rangle)$.
We also have that ${\phi}_{n-1}$ is an isomorphism between resonance digraphs $\overrightarrow{R}(G_{n-1})$ and $\overrightarrow{R}(G'_{n-1})$,
where $\phi_{n-1}$  is the restriction of $\phi$ on $R(G_{n-1})$.
Hence, $\phi_{n-1}$ maps  $\langle\mathcal{M}(G_{n-1}; uv)\rangle$ to $\langle\mathcal{M}(G'_{n-1}; u'v')\rangle$
such that if an edge $xy$ of $\langle\mathcal{M}(G_{n-1}; uv)\rangle$ is directed
from $x$ to $y$ in $\overrightarrow{R}(G_{n-1})$,
then $\phi_{n-1}(x)\phi_{n-1}(y)$ is an edge of $\langle\mathcal{M}(G'_{n-1}; u'v')\rangle$ directed
from $\phi_{n-1}(x)$ to $\phi_{n-1}(y)$ in $\overrightarrow{R}(G'_{n-1})$.

\smallskip
{\it Claim 1.}  Each edge of $ \mathcal{M}(G;P_n^+, {\partial s_n}) \rangle$ 
resulted from the peripheral convex expansion of an edge $x_1y_1$ in
$\langle\mathcal{M}(G_{n-1}; uv)\rangle$ has the same orientation as 
the edge of  $\mathcal{M}(G'; {P'}_n^+, {\partial s'_n}) \rangle$ 
resulted from the peripheral convex expansion of $\phi_{n-1}(x_1)\phi_{n-1}(y_1)$ in
$\langle\mathcal{M}(G'_{n-1}; u'v')\rangle$. 

{\it Proof of Claim 1. }
Let $x_1y_1$ be an edge in $\langle\mathcal{M}(G_{n-1}; uv)\rangle$.
Then $\phi_{n-1}(x_1)\phi_{n-1}(y_1)$ is its corresponding edge
under $\phi_{n-1}$ in $ \langle\mathcal{M}(G'_{n-1}; u'v')\rangle$.

Assume that $x_1x_2$ and $y_1y_2$ are two edges 
between $R(G_{n-1})$ and {$\langle \mathcal{M}(G;P_n^+, {\partial s_n}) \rangle$},
where both edges  have face-label $s_n$.
Then $x_2y_2$ is an edge of $\langle \mathcal{M}(G;P_n^+, {\partial s_n}) \rangle$ resulted from the peripheral convex expansion of the edge $x_1y_1$. Note that $\phi_{n-1}(x_1)\phi(x_2)$ and $\phi_{n-1}(y_1)\phi(y_2)$ are two edges 
between $R(G'_{n-1})$ and {$\langle \mathcal{M}(G';{P'}_n^+, {\partial s'_n}) \rangle$},
where both edges have face-label $s'_n=\widehat{\phi}(s_n)$.
Hence, $\phi(x_2)\phi(y_2)$ is an edge of $ \langle \mathcal{M}(G';{P'}_n^+, {\partial s'_n}) \rangle$ 
resulted from the peripheral convex expansion of the edge $\phi_{n-1}(x_1)\phi_{n-1}(y_1)$.

Without loss of generality, we show that if $x_2y_2$ is directed from $x_2$ to $y_2$ in $\overrightarrow{R}(G)$, 
then $\phi(x_2)\phi(y_2)$ is directed from $\phi(x_2)$ to $\phi(y_2)$ in $\overrightarrow{R'}(G)$.

Recall  both edges $x_1x_2$ and $y_1y_2$ of $R(G)$ have face-label $s_n$.
Then $x_1=x_2 \oplus \partial s_n$ and  $y_1=y_2 \oplus \partial s_n$. 
By the peripheral convex expansion structure of $R(G)$ from $R(G_{n-1})$, vertices $x_1,y_1, y_2, x_2$ form a 4-cycle $C$ in $R(G)$.
It is well known \cite{C18} that two antipodal edges of a 4-cycle in $R(G)$ have the same face-label
and two face-labels of adjacent edges of a 4-cycle in $R(G)$ are vertex disjoint faces of $G$.
Assume that two antipodal edges $x_1y_1$ and $x_2y_2$ of $C$ in $R(G)$ have the face-label $s_k$ for some $1 \le k \le n-1$.
Then $x_1 \oplus y_1 =x_2 \oplus y_2 =\partial s_k$ where $s_k$ is vertex disjoint from $s_n$.
By our assumption that $x_2y_2$ is directed from $x_2$ to $y_2$ in $\overrightarrow{R}(G)$, 
it follows that $x_1y_1$ is directed from $x_1$ to $y_1$ in $\overrightarrow{R}(G_{n-1}) \subset \overrightarrow{R}(G)$.

Since $\phi_{n-1}$ is an isomorphism between resonance digraphs $\overrightarrow{R}(G_{n-1})$ and  $\overrightarrow{R'}(G_{n-1})$,
we have that $\phi_{n-1}(x_1)\phi_{n-1}(y_1)$ is directed from $\phi_{n-1}(x_1)$ to $\phi_{n-1}(y_1)$ in $\overrightarrow{R'}(G_{n-1})$.
Similarly to the above argument,  vertices $\phi_{n-1}(x_1), \phi_{n-1}(y_1), \phi(y_2), \phi(x_2)$ form a 4-cycle in $R(G')$,
where two antipodal edges $\phi_{n-1}(x_1)\phi_{n-1}(y_1)$ and $\phi(x_2) \phi(y_2)$ of $C'$ in $R(G')$ 
have the face-label $s'_k=\widehat{\phi}_{n-1}(s_k)$, where $s'_k$ and $s'_n$ are vertex disjoint faces of $G'$.
Recall both edges $\phi_{n-1}(x_1)\phi(x_2)$ and $\phi_{n-1}(y_1)\phi(y_2)$ have face-label $s'_n=\widehat{\phi}(s_n)$.
Then $\phi(x_2)=\phi_{n-1}(x_1) \oplus \partial s'_n$ and  $\phi(y_2)=\phi_{n-1}(y_1) \oplus \partial s'_n$. 
It follows that $\phi(x_2)\phi(y_2)$ is directed from $\phi(x_2)$ to $\phi(y_2)$ in $\overrightarrow{R'}(G)$.
Therefore, Claim 1 holds.
\smallskip

{\it Claim 2.} The edges  between 
$\mathcal{M}(G;P_n^-, {\partial s_n})$ 
and $\mathcal{M}(G; P_n^{+}, \partial s_n)$ in  $\overrightarrow{R}(G)$ have the same orientation as the edges  between 
$\mathcal{M}(G'; {P'_n}^{-}, \partial s'_n)$ and $\mathcal{M}(G'; {P'_n}^{+}, \partial s'_n)$ in $\overrightarrow{R}(G')$.

{\it Proof of Claim 2.} By definitions of $\mathcal{M}(G;P_n^-, {\partial s_n})$, $\mathcal{M}(G; P_n^{+}, \partial s_n)$, 
and directed edges in $\overrightarrow{R}(G)$,
we can see that all edges between 
$\mathcal{M}(G;P_n^-, {\partial s_n})$ 
and $\mathcal{M}(G; P_n^{+}, \partial s_n)$
are directed from one set to the other.
Similarly, all edges between $\mathcal{M}(G'; {P'_n}^{-}, \partial s'_n)$ and $\mathcal{M}(G'; {P'_n}^{+}, \partial s'_n)$
are directed from one set to the other.

Let $M_{\hat{0}}$ be the minimum and $M_{\hat{1}}$ the maximum in the distributive lattice $\mathcal{M}(G_{n-1})$. 
By Lemma \ref{grozna},  exactly one of these two perfect matchings contains the edge $uv$. 
Without loss of generality,  let $M_{\hat{0}} \in \mathcal{M}(G_{n-1}; uv)$
where $\langle\mathcal{M}(G_{n-1}; uv)\rangle \cong \langle \mathcal{M}(G;P_n^-, {\partial s_n}) \rangle \cong \langle \mathcal{M}(G;P_n^+, {\partial s_n}) \rangle$. 
Let $\widehat{M_{\hat{0}}}$ be the perfect matching of $G$ such that $M_{\hat{0}} \setminus \{ uv \} \subseteq \widehat{M_{\hat{0}}}$. 
Then  $\widehat{M_{\hat{0}}} \in \mathcal{M}(G;P^+, {\partial s})$ is the minimum of the distributive lattice $\mathcal{M}(G)$.

Let $M'_{\hat{0}}=\phi_{n-1}(M_{\hat{0}})$. Then $M'_{\hat{0}}$ is the minimum of the distributive lattice $\mathcal{M}(G'_{n-1})$,
and $M'_{\hat{0}} \in \mathcal{M}(G'_{n-1}; u'v')$ 
where $\langle\mathcal{M}(G'_{n-1}; u'v')\rangle   \cong \langle \mathcal{M}(G';{P'}_n^-, {\partial s'_n}) \rangle \cong \langle \mathcal{M}(G';{P'}_n^+, {\partial s'_n}) \rangle$.
As before, define  $\widehat{M'_{\hat{0}}}$ as the perfect matching of $G'$ such that $M'_{\hat{0}} \setminus \{ u'v' \} \subseteq \widehat{M'_{\hat{0}}}$. 
By Lemma \ref{grozna},  $\widehat{M'_{\hat{0}}} \in \mathcal{M}(G';P'^+, {\partial s})$ is the minimum of the distributive lattice $\mathcal{M}(G')$.
This implies that Claim 2 holds.
\smallskip
}

Consequently, $\phi$ is {also} an isomorphism between resonance digraphs  $\overrightarrow{R}(G)$ and $\overrightarrow{R}(G')$,  
which means that property $(iii)$ holds.

Suppose that $P_n=\partial s_n -uv$ starts with $u$ and ends with $v$ along  the clockwise orientation of the periphery of $G$, 
and $P'_n=\partial s'_n -u'v'$ starts with $u'$ and ends with $v'$ along the  the clockwise orientation of the periphery of $G'$. 
Since the resonance digraphs $\overrightarrow{R}(G)$ and $\overrightarrow{R}(G')$ are isomorphic, it follows that $u$ and $u'$ have the same color 
and $v$ and $v'$ have the same color. 
So, the above reducible face decompositions of $G$ and $G'$ also satisfy property $(ii)$  when  $i = n$. 
Therefore, property $(ii)$ holds. \qed
\end{proof}
\bigskip

The following corollary follows directly from Theorem \ref{glavni}.

\begin{corollary} \label{cor1} Let $G$ and $G'$ be 2-connected outerplane bipartite graphs. 
Then their resonance graphs  $R(G)$ and $R(G')$ are isomorphic if and only if 
 $G$ and $G'$ can be properly two colored so that $\overrightarrow{R}(G)$ and $\overrightarrow{R}(G')$ are isomorphic.
\end{corollary}

We are now ready to state the following main result of the paper.

\begin{theorem} \label{glavni1} Let $G$ and $G'$ be 2-connected outerplane bipartite graphs. 
Then their resonance graphs  $R(G)$ and $R(G')$ are isomorphic if and only if 
$G$ has a reducible face decomposition $G_i (1 \le i \le n)$
associated with the face sequence 
$s_i (1 \le i \le n)$ and the odd length path sequence $P_i (2 \le i  \le n)$;  and
$G'$ has a reducible face decomposition $G'_i (1 \le i \le n)$ 
associated with the face sequence $s'_i (1 \le i  \le n)$ and the odd length path sequence $P'_i (2 \le i  \le n)$ satisfying two properties:
\begin{itemize}
\item [$(i)$] the map sending $s_i$ to $s'_i$ induces an isomorphism between the inner dual of $G$ and inner dual of $G'$  for $1 \le i  \le n$; 
and
\item [$(ii)$] $G$ and $G'$ can be properly two colored so that odd length paths $P_i$ and $P'_i$ 
either both start from a black vertex and end with a white vertex, or both start from 
a white vertex and end with a black vertex in  clockwise orientation along the peripheries of $G_i$ and $G'_i$ for $2 \le i  \le n$.
\end{itemize}
\end{theorem}

\begin{proof}
{\it Necessity.} This implication follows by Theorem \ref{glavni}.

{\it Sufficiency.} Let $G$ and $G'$ be 2-connected outerplane bipartite graphs each with $n$ inner faces. 
Use induction on $n$.
The result holds when $n=1$ or $2$.  Assume that $n \ge 3$ and the result holds for 
any two 2-connected outerplane bipartite graphs each with less than $n$ inner faces.
By Theorem \ref{T:OuterPlanePeripheryExpansionSequence},
 $R(G)$ can be obtained from an edge by a peripheral convex expansions  with respect to a reducible face decomposition $G_i (1 \le i \le n)$
associated with the face sequence
$s_i (1 \le i \le n)$ and the odd length path sequence $P_i (2 \le i  \le n)$;  and
$R(G')$ can be obtained from an edge by a peripheral convex expansions  with respect to a reducible face decomposition $G'_i (1 \le i \le n)$ 
associated with the face sequence $s'_i (1 \le i  \le n)$ and the odd length path sequence $P'_i (2 \le i  \le n)$.

{Assume that properties $(i)$ and $(ii)$ hold for the above reducible face decompositions of $G$ and $G'$.
By induction hypothesis, $R(G_{n-1})$ and $R'(G_{n-1})$ are isomorphic.}

Similarly to the argument in Theorem \ref{glavni},
we can see that $s_n$ is adjacent to exactly one inner face $s_j$ of $G$ such that $\partial s_n \cap \partial s_{j}$ is an edge $uv$ on $\partial G_{n-1}$;
and $s'_n$  is adjacent to exactly one inner face $s'_j$ of $G'$ such that $\partial s'_n \cap \partial s'_{j}$ is an edge $u'v'$ on $\partial G'_{n-1}$.
It is clear that $u$ and $v$ (resp., $u'$ and $v'$) are two end vertices of $P_n$ (resp., $P'_n$).
Moreover,  
$R(G)=pce(R(G_{n-1}), \langle\mathcal{M}(G_{n-1}; uv)\rangle)$,
and $R(G')=pce(R(G'_{n-1}), \langle\mathcal{M}(G'_{n-1}; u'v')\rangle)$. 
To show that $R(G)$ and $R(G')$ are isomorphic, it remains to prove that $\langle\mathcal{M}(G_{n-1}; uv)\rangle$
and $\langle\mathcal{M}(G'_{n-1}; u'v')\rangle$ are isomorphic.

Let $H_{n-1}$ be the subgraph of $G_{n-1}$ obtained by removing two end vertices of the edge $uv$, 
and repeatedly removing end vertices of resulted pendant edges during the process,
and $H'_{n-1}$ be the subgraph of $G'_{n-1}$ obtained by removing two end vertices of the edge $u'v'$, 
and repeatedly removing end vertices of resulted pendant edges during the process.
Note that all vertices of $G_{n-1}$ (resp., $G'_{n-1}$)
are on the outer cycle of $G_{n-1}$ (resp., $G'_{n-1}$) since $G_{n-1}$ (resp., $G'_{n-1}$) is an outerplane graph.  
Then all resulted pendant edges during the process of obtaining $H_{n-1}$ (resp., $H'_{n-1}$) 
from $G_{n-1}$ (resp., $G'_{n-1}$) are the edges
on the outer cycle of $G_{n-1}$ (resp., $G'_{n-1}$).
It follows  either both $H_{n-1}$ and $H'_{n-1}$ are empty,  
or $H_{n-1}$ and $H'_{n-1}$ are connected subgraphs of $G_{n-1}$ and $G'_{n-1}$, respectively.
Moreover, if both $H_{n-1}$ and $H'_{n-1}$ are empty, then $\mathcal{M}(G_{n-1}; uv )$ contains
a unique perfect matching of $G_{n-1}$ and $\mathcal{M}(G'_{n-1}; u'v' )$ contains a unique perfect matching of $G'_{n-1}$. 
So, $\langle \mathcal{M}(G_{n-1}; uv )\rangle$
and $ \langle \mathcal{M}(G'_{n-1}; u'v') \rangle$ are isomorphic as single vertices.

We now assume that $H_{n-1}$ and $H'_{n-1}$ are connected subgraphs of $G_{n-1}$ and $G'_{n-1}$, respectively. 
It is clear that all resulted pendant edges during the process of obtaining $H_{n-1}$ from $G$ are the edges
 of each perfect matching in $\mathcal{M}(G_{n-1}; uv )$, and so each perfect matching of 
 $H_{n-1}$ can be extended uniquely to a perfect matching  in  $\mathcal{M}(G_{n-1}; uv )$.
Hence, there is a 1-1 correspondence between the set of perfect matchings of  $H_{n-1}$ and the set of perfect matchings  in 
 $\mathcal{M}(G_{n-1}; uv)$.
 Two perfect matchings of  $H_{n-1}$ are adjacent in  $R(H_{n-1})$ if and only if the corresponding two perfect matchings  in 
 $\mathcal{M}(G_{n-1}; uv )$ are adjacent in $\langle \mathcal{M}(G_{n-1}; uv) \rangle$.
Therefore, $R(H_{n-1})$  is isomorphic to $\langle \mathcal{M}(G_{n-1}; uv)\rangle$.
Similarly, $R(H'_{n-1})$ is isomorphic to $\langle \mathcal{M}(G'_{n-1}; u' v') \rangle$.

{Next, we show that $R(H_{n-1})$  and $R(H'_{n-1})$ are isomorphic.
Note that $G_{n-1}$ and $G'_{n-1}$ have reducible face decompositions satisfying properties $(i)$ and $(ii)$.
By the constructions of $H_{n-1}$ and $H'_{n-1}$, we can distinguish two cases based on
whether $H_{n-1}$ and $H'_{n-1}$ are 2-connected or not.

\textit{Case 1.} $H_{n-1}$ and $H'_{n-1}$ are 2-connected. 
Then by their constructions, $H_{n-1}$ and $H'_{n-1}$ have reducible face decompositions satisfying properties $(i)$ and $(ii)$.
 Then $R(H_{n-1})$  and $R(H'_{n-1})$ are isomorphic by induction hypothesis.

\textit{Case 2.}  Each of $H_{n-1}$ and $H'_{n-1}$ has more than one 2-connected component. 
Note that any  2-connected component of $H_{n-1}$ and $H'_{n-1}$ is 
a 2-connected outerplane bipartite graph.
This implies that any bridge of $H_{n-1}$ (resp., $H'_{n-1}$)  cannot 
belong to any perfect matching of
$H_{n-1}$ (resp., $H'_{n-1}$). 
Hence, any perfect matching of $H_{n-1}$ (resp., $H'_{n-1}$)  is the perfect matching of the union of its 2-connected components.
It follows that $R(H_{n-1})$ (resp., $R(H'_{n-1})$) is  the Cartesian product of resonance graphs of its  2-connected components.
By the constructions of $H_{n-1}$ and $H'_{n-1}$, there is a 1-1 corresponding between the set of 2-connected components of $H_{n-1}$ 
and the set of 2-connected components of $H'_{n-1}$ 
such that each 2-connected component of $H_{n-1}$ and its corresponding 2-connected component of $H'_{n-1}$
have reducible face decompositions satisfying properties $(i)$ and $(ii)$.
Then $R(H_{n-1})$  and $R(H'_{n-1})$ are isomorphic by induction hypothesis.}

We have shown that $R(H_{n-1})$  is isomorphic to $\langle \mathcal{M}(G_{n-1}; uv )\rangle$,
 and $R(H'_{n-1})$ is isomorphic to $\langle\mathcal{M}(G'_{n-1}; u'v' )\rangle$.
Therefore, $\langle \mathcal{M}(G_{n-1}; uv )\rangle$
and $\langle \mathcal{M}(G'_{n-1}; u' v')\rangle$ are isomorphic.
It follows that $R(G)$ and $R(G')$ are isomorphic. \qed

\end{proof}

\bigskip


Finally, we can formulate the main result  using local structures of given graphs. 
 Let $e$ and $f$ be two edges of a graph $G$. Let $d_G(e,f)$ denote the distance between corresponding vertices in the line graph $L(G)$ of $G$.
The following concepts introduced in \cite{br-tr-zi-2} will be also needed for that purpose.
\begin{definition} \cite{br-tr-zi-2}
Let $G$ be a 2-connected outerplane bipartite graph and $s$, $s'$, $s''$ be three inner faces of  $G$. 
Then the triple $(s,s',s'')$ is called an \textit{adjacent triple of inner faces} if $s$ and $s'$ have the common edge $e$ and $s',s''$ have the common edge $f$.  
The adjacent triple of  inner faces $(s,s',s'')$ is \textit{regular} if the distance $d_G(e,f)$ is an even number, and \textit{irregular} otherwise. 
\end{definition}

It is easy to see that 2-connected outerplane bipartite graphs $G$ and $G'$ have  reducible face decompositions satisfying properties $(i)$ and $(ii)$ 
 if and only if there exists an isomorphism between the inner dual of $G$ and $G'$ that preserves the (ir)regularity of adjacent triples of inner faces. Therefore, the next result follows directly from Theorem \ref{glavni1}. 

\begin{corollary} \label{cor2}
Let $G$ and $G'$ be two 2-connected outerplane bipartite graphs with inner duals $T$ and $T'$, respectively. Then $G$ and $G'$ have isomorphic resonance graphs if and only if there exists an isomorphism $\alpha: V(T) \rightarrow V(T')$ such that for any 3-path $xyz$ of $T$: the adjacent triple  $(x,y,z)$ of inner faces of $G$ is regular if and only if  the adjacent triple $(\alpha(x),\alpha(y),\alpha(z))$ of inner faces  of $G'$ is regular.
\end{corollary}

Since it would be interesting to generalize the presented results to a wider family of graphs (for example plane elementary bipartite graphs), we conclude the paper with the following open problem.
\bigskip

\noindent
\textbf{Problem 2.} \textit{Characterize plane (elementary) bipartite graphs   with isomorphic resonance graphs.}
\bigskip
\bigskip

\noindent{\bf Acknowledgment:} Simon Brezovnik, Niko Tratnik, and Petra \v Zigert Pleter\v sek acknowledge the financial support from the Slovenian Research Agency: research programme No.\ P1-0297 (Simon Brezovnik, Niko Tratnik, Petra \v Zigert Pleter\v sek), project No.\ N1-0285 (Niko Tratnik), and project No.\ NK-0001 (Petra \v Zigert Pleter\v sek).  All
four authors thank the Slovenian Research Agency for financing our bilateral project between Slovenia and the USA (title: \textit{Structural properties of resonance graphs and related concepts}, project No. BI-US/22-24-158).
  

\end{document}